\documentclass{article}
\usepackage{amssymb}
\usepackage{amsthm}
\usepackage[curve,matrix,arrow]{xy}
\theoremstyle{plain}\newtheorem{Theorem}{Theorem}[section]
\theoremstyle{plain}
\theoremstyle{plain}
\theoremstyle{plain}\newtheorem{Lemma}[Theorem]{Lemma}
\theoremstyle{plain}\newtheorem{Proposition}[Theorem]{Proposition}
\theoremstyle{plain}\newtheorem{statement}[Theorem]{}
\theoremstyle{definition}\newtheorem{Definition}[Theorem]{Definition}
\theoremstyle{definition}\newtheorem{Example}[Theorem]{Example}
\theoremstyle{definition}
\theoremstyle{definition}\newtheorem{Notation}[Theorem]{Notation}
\theoremstyle{definition}\newtheorem{Remark}[Theorem]{Remark}
\theoremstyle{definition}

    \def\OP{{\mathcal{O}P}}

\def\CO{{\mathcal{O}}}

\def\Aut{\mathrm{Aut}}           \def\tenk{\otimes_k}     
             \def\ten{\otimes}

\def\dim{\mathrm{dim}}           
\def\End{\mathrm{End}}           
\def\Endbar{\underline{\mathrm{End}}}
\def\Ext{\mathrm{Ext}}

\def\Hom{\mathrm{Hom}}           
\def\Hombar{\underline{\mathrm{Hom}}}

\def\ker{\mathrm{ker}}           
\def\Id{\mathrm{Id}}             \def\tenA{\otimes_A}
\def\Im{\mathrm{Im}}             \def\tenB{\otimes_B}
           
\def\Inn{\mathrm{Inn}}

           \def\tenO{\otimes_{\mathcal{O}}}
\def\mod{\mathrm{mod}}

\def\op{\mathrm{op}}
\def\Out{\mathrm{Out}}
\def\Pic{\mathrm{Pic}}
\def\pr{\mathrm{pr}}

\def\st{\mathrm{st}}      
\def\stmod{\mathrm{stmod}}

\title{Integrable derivations and stable
equivalences of Morita type} 
\author{Markus Linckelmann} 
 \date{}

\begin{document}

\maketitle

\begin{abstract}
Using that integrable derivations of symmetric algebras can
be interpreted in terms of Bockstein homomorphisms in Hochschild
cohomology, we show that integrable derivations are 
invariant under the transfer maps in Hochschild cohomology of
symmetric algebras induced by stable equivalences of Morita type. 
With applications in block theory in mind, we allow complete
discrete valuation rings of unequal characteristic.
\end{abstract}

\section{Introduction}

Throughout this paper, $\CO$ is a complete discrete valuation ring with 
maximal ideal $J(\CO)=$ $\pi\CO$ for some nonzero element $\pi\in$ 
$\CO$, residue field $k=$ $\CO/J(\CO)$, and field of fractions $K$. 
Let $A$ be an $\CO$-algebra such that $A$ is free of finite rank as an 
$\CO$-module. Let $r$ be a positive integer and let $\alpha$ be an 
$\CO$-algebra automorphism of $A$ such that $\alpha$ induces the 
identity map on $A/\pi^rA$; that is, for all $a\in$ $A$ we have 
$\alpha(a)=$ $a + \pi^r\mu(a)$ for some $\CO$-linear endomorphism $\mu$ 
of $A$. It is well-known, and easy to verify, that the endomorphism of 
$A/\pi^rA$ induced by $\mu$ is a derivation on $A/\pi^rA$. Any 
derivation on $A/\pi^rA$ which arises in this way is called {\it 
integrable}; this concept goes back to work of Gerstenhaber \cite{Ger1}
(see \S \ref{derSection} for more details). The set of 
integrable derivations on $A/\pi^rA$ is an abelian group containing all 
inner derivations, hence determines a subgroup of $HH^1(A/\pi^rA)$, 
denoted $HH^1_A(A/\pi^rA)$. It is shown in \cite[\S 4]{FGM} that if 
$\CO=$ $k[[t]]$, $r=1$, and if $A=$ $\bar A\tenk k[[t]]$ for some 
finite-dimensional $k$-algebra $\bar A$, then $HH^1_A(\bar A)$ is 
invariant 
under Morita equivalences between finite-dimensional $k$-algebras. We 
show that this invariance extends to stable equivalences of Morita type 
between symmetric algebras over an arbitrary complete discrete 
valuation ring $\CO$. An $\CO$-algebra $A$ is called {\it symmetric} if 
$A$ is $\CO$-free of finite rank and if $A\cong$ $A^\vee=$ 
$\Hom_\CO(A,\CO)$ as $A\tenO A^\op$-modules. Following Brou\'e 
\cite[\S 5 A]{BroueEq}, given symmetric $\CO$-algebras $A$ and $B$, an 
$A$-$B$-bimodule $M$ is said to induce a {\it stable equivalence of 
Morita type}, if $M$ is finitely generated projective as a left 
$A$-module and as a right $B$-module, and if the bimodules 
$M\tenB M^\vee$ and $M^\vee\tenA M$ are isomorphic to $A$ and $B$ in the 
relatively $\CO$-stable categories of $A\tenO A^\op$-modules and 
$B\tenO B^\op$-modules, respectively. 

\begin{Theorem} \label{integrableinvariant}
Let $A$ and $B$ be symmetric $\CO$-algebras such that the semisimple 
quotients of $A$ and $B$ are separable. Let $r$ be a positive integer.
Suppose that the canonical maps $Z(A)\to$ $Z(A/\pi^rA)$ and $Z(B)\to$ 
$Z(B/\pi^rB)$ are surjective. Let $M$ be an $A$-$B$-bimodule inducing 
a stable equivalence of Morita type. Then the functor 
$M\tenB - \tenB M^\vee$ induces an $\CO$-linear isomorphism 
$HH^1(B/\pi^rB)\cong$ $HH^1(A/\pi^rA)$ which restricts to an
isomorphism of groups 
$$HH^1_B(B/\pi^rB)\cong HH^1_A(A/\pi^rA). $$
In particular, if $A$ and $B$ are derived or Morita equivalent, 
then $HH^1_B(B/\pi^rB)\cong$ $HH^1_A(A/\pi^rA)$.
\end{Theorem}

By a result of Rickard \cite[5.5]{RickDer}, a derived equivalence 
between symmetric algebras induces a stable equivalence of Morita type, 
and hence the last statement in the theorem is an immediate consequece
of the first.  
Theorem \ref{integrableinvariant} will be proved, slightly more 
generally for relatively $\CO$-selfinjective algebras, in Theorem 
\ref{outmstable}. 

\begin{Remark}
The assumption that $Z(A)\to$ $Z(A/\pi^rA)$ is surjective holds if $A$ 
is a finite group algebra, or a block algebra, or a source algebra of 
a block. It also holds in the `classic' case where $\CO=$ $k[[t]]$ and 
where $A$ is of the form $\bar A\tenk k[[t]]$ for some 
finite-dimensional $k$-algebra $\bar A$. In that case, $HH^1_A(\bar A)$ 
is a $Z(\bar A)$-submodule, and a stable equivalence of Morita type 
between $\bar A$ and another finite-dimensional $k$-algebra $\bar B$ 
extends to a stable equivalence of Morita type between the 
$k[[t]]$-algebras $A=$ $\bar A\tenk k[[[t]]$ and $B=$ 
$\bar B\tenk k[[t]]$. By contrast, in the more general situation of the 
above theorem, allowing $k$ and $K$ to have unequal characteristic, the 
subgroup $HH^1_A(A/\pi^rA)$ need not be an $\CO/\pi^r\CO$-submodule of 
$HH^1(A/\pi^r A)$, and not every stable equivalence of Morita type 
between $A/\pi^rA$ and $B/\pi^rB$ lifts necessarily to a stable 
equivalence of Morita type between $A$  and $B$. See \S 
\ref{Examplessection} for some examples. 
\end{Remark}

\begin{Notation}\label{aut-notation}
In what follows, the use of algebra automorphisms as subscripts
to modules is as in \cite{Listable}. If $\alpha$ is an automorphism of 
an $\CO$-algebra $A$ and $U$ an $A$-module, we denote by ${_\alpha{U}}$ 
the $A$-module which is equal to $U$ as an $\CO$-module, with $a\in$ $A$
acting as $\alpha(a)$ on $U$. If $\alpha$ is inner, then 
${_\alpha{U}}\cong$ $U$; indeed, if $c\in$ $A^\times$ such that 
$\alpha(a)=$ $cac^{-1}$ for all $a\in$ $A$, then the map sending $u\in$ 
$U$ to $cu$ is an $A$-module isomorphism $U\cong$ ${_{\alpha}{U}}$.
We use the analogous notation for right modules and 
bimodules. If $U$ and $V$ are $A$-$A$-bimodules and $\alpha\in$
$\Aut(A)$, then we have an obvious isomorphism of $A$-$A$-bimodules
$(U_{\alpha})\tenA V\cong$ $U\tenA ({_{\alpha^{-1}}{V}})$.
\end{Notation}

\section{Background material}

We collect in this section some basic and well-known facts on 
Hochschild cohomology and stable categories. 
Let $A$ be an $\CO$-algebra such that $A$ is free of finite rank as an 
$\CO$-module. For any integer $n\geq 0$ and any $A\tenO A^\op$-module 
$M$, the Hochschild cohomology in degree $n$ of $A$ with coefficients 
in $M$ is $HH^n(A;M)=$ $\Ext^n_{A\tenO A^\op}(A;M)$; we set $HH^n(A)=$
$HH^n(A;A)$. (We use here that $A$ is $\CO$-free, since in general, 
Hochschild cohomology is defined as a relative $\Ext$-module.) We have 
$HH^0(A)\cong$ $Z(A)$, and $HH^1(A)$ is the space of derivations on $A$ 
modulo inner derivations. The graded $\CO$-module $HH^*(A)=$ 
$\oplus_{n\geq 0}\ HH^n(A)$ is a graded-commutative algebra, and the 
positive degree part is a graded Lie algebra of degree $-1$, with the 
Gerstenhaber bracket; in particular, $HH^1(A)$ is a Lie algebra. See 
e. g. \cite[Ch. 9]{Weibel} for more  material on Hochschild cohomology.  
For $U$, $V$ two $\CO$-free $A$-modules and any integers $r$, $s$
such that $s\geq r>0$ we 
have an obvious isomorphism $\Hom_A(U,V/\pi^rV)\cong$ 
$\Hom_{A/\pi^sA}(U/\pi^s U, V/\pi^r V)$. If $P$ is a projective
resolution of $U$, then $P/\pi^s P$ is a projective resolution of the 
$A/\pi^sA$-module $U/\pi^sU$, and we have an isomorphism of cochain 
complexes $\Hom_A(P, V/\pi^r V)\cong$ 
$\Hom_{A/\pi^s A}(P/\pi^s P, V/\pi^r V)$. Taking cohomology in degree 
$n\geq$ $0$ yields an isomorphism $\Ext^n_A(U,V/\pi^r V)\cong$
$\Ext^n_{A/\pi^s A}(U/\pi^s U,V/\pi^r V)$. In particular, for any 
$r$, $s$ such that $s\geq r >0$ and any $n\geq$ $0$ we have 

\begin{statement} \label{HHnrs}
$$HH^n(A;A/\pi^r A)\cong HH^n(A/\pi^sA;A/\pi^rA) 
\cong HH^n(A/\pi^rA)$$ 
\end{statement}

In what follows the superscript $\pi^r$ to an arrow means the map given 
by multiplication with $\pi^r$. 
We refer to \cite{Brow} for background material on Bockstein 
homomorphisms and spectral sequences. 
Applying the functor 
$\Hom_{A\tenO A^\op}(A,-)$ to the Bockstein short exact sequence of 
$A$-$A$-bimodules
$$\xymatrix{0\ar[r]&A\ar[r]^{\pi^r}&A\ar[r]&A/\pi^r A\ar[r]&0}$$
and making use of the above identifications
yields a long exact sequence of $\CO$-modules which 
starts as follows:

\begin{statement} \label{ZAZbarAlongexact}
$$\xymatrix{0\ar[r]&Z(A)\ar[r]^{\pi^r}&Z(A)\ar[r]&Z(A/\pi^r A)\ar[r]
&HH^1(A)\ar[r]^{\pi^r}&HH^1(A)\ar[r]&}$$
$$\xymatrix{&HH^1(A/\pi^r A)\ar[r]&HH^2(A)\ar[r]^{\pi^r}
&HH^2(A)\ar[r]&\cdots}$$
\end{statement}

If $r$ and $n$ are positive integers
such that that $\pi^r$ annihilates $HH^n(A)$ and $HH^{n+1}(A)$,
then it follows from the long exact sequence \ref{ZAZbarAlongexact}
that there is a short exact sequence

\begin{statement} \label{HHlemma1}
$$\xymatrix{0\ar[r]&HH^n(A)\ar[r]&HH^n(A/\pi^r A)\ar[r]&HH^{n+1}(A)
\ar[r] & 0}$$
\end{statement}

For future use, we briefly note some further immediate 
consequences of \ref{ZAZbarAlongexact}.

\begin{Lemma} \label{ZAZbarAHH1}
Let $r$ be a positive integer.

\smallskip\noindent (i)
The canonical map $Z(A)\to$ $Z(A/\pi^rA)$ is surjective if and
only if $HH^1(A)$ is $\CO$-free.

\smallskip\noindent (ii)
The connecting homomorphism $HH^1(A/\pi^rA)\to$ $HH^2(A)$ is
injective if and only if $HH^1(A)=$ $\{0\}$.
\end{Lemma}

\begin{proof}
The above long exact sequence implies that the map $Z(A)\to$
$Z(A/\pi^r A)$ is surjective if and only if the map
$Z(A/\pi^r A)\to$ $HH^1(A)$ is zero, hence if and only if the map 
$HH^1(A)\to$ $HH^1(A)$ induced by multiplication with $\pi^r$ is 
injective. This is the case if and only if $HH^1(A)$ is $\CO$-free, 
proving (i). Similarly, the connecting homomorphism 
$HH^1(A/\pi^r A)\to$ $HH^2(A)$ is injective if and only if the map 
$HH^1(A)\to$ $HH^1(A/\pi^r A)$ is zero, hence if and only if the map 
$HH^1(A)\to$ $HH^1(A)$ induced by multiplication by $\pi^r$ is 
surjective. By Nakayama's Lemma, this is equivalent to $HH^1(A)$ being 
zero, which proves (ii).
\end{proof}

The property of $HH^1(A)$ being $\CO$-free or zero does not depend on 
the integer $r$ in the above lemma. It follows therefore from the first 
statement of the lemma that if $Z(A)\to$ $Z(A/\pi^rA)$ is surjective for 
some positive integer $r$, then it is surjective for any positive 
integer $r$. Similarly, the second statement of this lemma implies that 
if $HH^1(A/\pi^r A)\to$ $HH^2(A)$ is injective for some positive integer 
$r$, then it is injective for any positive integer $r$. The following
observation is well-known.

\begin{Lemma}\label{KAseparable}
Suppose that $K\tenO A$ is separable. Then $HH^n(A)$ is
a torsion $\CO$-module for all positive integers $n$.
\end{Lemma}

\begin{proof}
Since $K\tenO A$ is separable it follows that $HH^n(K\tenO A)$
is zero for all positive integers $n$. The exactness of the
functor $K\tenO -$ implies that $\{0\}=$ $HH^n(K\tenO A)\cong$ 
$K\tenO HH^n(A)$. Since $K\tenO -$ annihilates exactly the torsion 
$\CO$-modules, the result follows.
\end{proof}

\begin{Lemma} \label{HHlemma5}
Suppose that $K\tenO A$ is separable. Let $r$ be a positive
integer. The following are equivalent.

\smallskip\noindent (i)
The canonical map $Z(A)\to$ $Z(A/\pi^r A)$ is surjective.

\smallskip\noindent (ii)
We have $HH^1(A)=$ $\{0\}$.

\smallskip\noindent (iii)
The connecting homomorphism $HH^1(A/\pi^r A)\to$ $HH^2(A)$ is
injective.
\end{Lemma}

\begin{proof}
If (i) holds, then the two previous lemmas imply that $HH^1(A)$ 
is both $\CO$-torsion and torsion free, hence zero.
The rest follows from \ref{ZAZbarAHH1}.
\end{proof}

As before, if (i) or (iii) in this lemma holds for some positive 
integer $r$, it holds for all positive integers $r$, since (ii)
does not depend on $r$. This implies further that if the
equivalent statements in the lemma hold, then all canonical maps 
$Z(A/\pi^rA)\to$ $Z(A/\pi^{r+1}A)$ are surjective.

\begin{Lemma} \label{HHlemma6}
Suppose that $K\tenO A$ is separable and that the canonical map
$Z(A)\to$ $Z(A/\pi A)$ is surjective. Let $r$ be a positive integer.
The bimodule homomorphism $A/\pi^rA\to$ $A/\pi^{r+1}A$ given by
multiplication with $\pi$ on $A$ induces an injective map
$HH^1(A/\pi^rA)\to$ $HH^1(A/\pi^{r+1}A)$ making the diagram
$$\xymatrix{HH^1(A/\pi^r A) \ar[d] \ar[r] 
& HH^1(A/\pi^{r+1}A) \ar[d] \\ HH^2(A) \ar@{=}[r] & HH^2(A)}$$
commutative, where the vertical arrows are the connecting
homomorphisms. Moreover, if $\pi^r$ annihilates $HH^2(A)$, then
all maps in this diagram are isomorphisms. In particular, $HH^2(A)$
has an ascending finite filtration of subspaces isomorphic to 
$HH^1(A/\pi^iA)$, with $i\geq 1$.
\end{Lemma}

\begin{proof}
The commutativity of the diagram follows from comparing the long
exact sequences from \ref{ZAZbarAlongexact} for $r$ and $r+1$.
Since the canonical map $Z(A)\to$ $Z(A/\pi A)$ is surjective
and since $K\tenO A$ is separable, it follows from
\ref{HHlemma5} that the connecting homomorphisms in the statement
are injective, and hence that the map $HH^1(A/\pi^rA)\to$ 
$HH^1(A/\pi^{r+1}A)$ is injective.  If $\pi^r$ annihilates $HH^2(A)$, 
then the long exact sequence \ref{ZAZbarAlongexact} 
implies that the connecting homomorphism $HH^1(A/\pi^rA)\to$
$HH^2(A)$ is also surjective, whence the result.
\end{proof}

We conclude this section with a very brief review of relatively
$\CO$-stable categories and stable equivalences; more details
can be found for instance in \cite{BroueHigman}. Let $A$ be
an $\CO$-algebra which is free of finite rank as an $\CO$-module.

An $A$-module $U$ is called {\it relatively $\CO$-projective} if
$U$ is isomorphic to a direct summand of $A\tenO V$ for some
$\CO$-module $V$. If $U$ is relatively $\CO$-projective, then the 
canonical map $A\tenO U\to$ $U$ sending $a\ten u$ to $au$ splits (this 
is a standard fact on the splitting of adjunction units and counits).
Dually, $U$ is called {\it relatively $\CO$-injective} if $U$ is
isomorphic to a direct summand of $\Hom_\CO(A,V)$ for some $\CO$-module
$V$, where the $A$-module structure on $\Hom_\CO(A,V)$ is given
by $(a\cdot\varphi)(b)=$ $\varphi(ba)$, for $a$, $b\in$ $A$ and
$\varphi\in$ $\Hom_\CO(A,V)$. As before, if $U$ is relatively
$\CO$-injective, then the canonical map $U\to$ $\Hom_\CO(A,U)$ 
sending $u\in$ $U$ to the map $a\mapsto$ $au$ is split. 

Suppose 
that $k\tenO A$ is selfinjective. Then $k\tenO A$ is injective as a 
left and right $k\tenO A$-module, hence its $k$-dual is a progenerator 
as a left and right $k\tenO A$-module. Since finitely generated
projective $k\tenO A$-modules lift uniquely, up to isomorphism, to 
finitely generated projective $A$-modules, it follows that the
$\CO$-dual $A^\vee=$ $\Hom_\CO(A,\CO)$ is a progenerator of
$A$ as a left and right $A$-module. Thus $U$ is relatively
$\CO$-projective if and only if $U$ is isomorphic to a direct
summand of finitely many copies of $A^\vee\tenO U=$
$\Hom_\CO(A,\CO)\tenO U\cong$ $\Hom_\CO(A,U)$, where the last
map sends $\alpha\ten u$ to the map $a\mapsto$ $\varphi(a)u$, for
$a\in$ $A$, $u\in$ $U$, $\alpha\in$ $\Hom_\CO(A,\CO)$. Thus
$U$ is relatively $\CO$-projective if and only if $U$ is relatively
$\CO$-injective. Identifying the relatively $\CO$-projective
modules to zero yields therefore a triangulated category,
called the {\it relatively $\CO$-stable category of finitely generated
$A$-modules $\stmod(A)$}. The objects of $\stmod(A)$ are the finitely 
generated $A$-modules. For any two finitely generated $A$-modules $U$, 
$U'$, denote by $\Hom_A^\pr(U,U')$ the set of $A$-homomorphisms from 
$U$ to $U'$ which factor through a relatively $\CO$-projective 
$A$-module. This is an $\CO$-submodule of $\Hom_A(U,U')$.
The space of morphisms in $\stmod(A)$ from $U$ to $U'$ is
the $\CO$-module $\Hombar_A(U,U')=$ $\Hom_A(U,U')/\Hom_A^\pr(U,U')$,
with composition of morphisms induced by that in $\mod(A)$.
The triangulated structure in $\stmod(A)$ is induced by $\CO$-split
short exact sequences of finitely generated $A$-modules.

Suppose that $A$ and $B$ are two $\CO$-algebras which are $\CO$-free 
of finite rank such that $k\tenO A$ and $k\tenO B$ are selfinjective.
Then $k\tenO (A\tenO B)$ and $k\tenO A^\op$ are selfinjective as well.
Let $M$ be an $A$-$B$-bimodule such that $M$ is finitely generated
projective as a left $A$-module and as a right $B$-module. The 
functor $M\tenB-$ from $\mod(B)$ to $\mod(A)$ is exact and preserves 
the classes of relatively $\CO$-projective modules, hence induces a 
functor from $\stmod(B)$ to $\stmod(A)$ as triangulated categories.
Let $N$ be a $B$-$A$-bimodule which is finitely generated as a left
$B$-module and as a right $A$-module. Following Brou\'e \cite{BroueEq}
we say that {\it $M$ and $N$ induce a stable equivalence of Morita type
between $A$ and $B$} if we have isomorphisms $M\tenB N\cong$ $B\oplus Y$ 
and $N\tenA M\cong$ $A\oplus X$ as $B\tenO B^\op$-modules and
$A\tenO A^\op$-modules, respectively, such that $Y$ is a projective
$B\tenO B^\op$-module and $X$ is a projective $A\tenO A^\op$-module.
In that case, the functors $M\tenB-$ and $N\tenA -$ induces
equivalences between $\stmod(B)$ and $\stmod(A)$ which are inverse
to each other. Moreover, for any positive integer $r$ the bimodules
$M/\pi^rM$ and $N/\pi^rN$ induce a stable equivalence of Morita
type between the $\CO/\pi^r\CO$-algebras $A/\pi^rA$ and $B/\pi^rB$.
If $A$ is symmetric, then $k\tenO A$ is selfinjective.
If also $B$ is symmetric and if as before $M$ and $N$ induce a stable
equivalence of Morita type between $A$ and $B$, then $M$ and its
$\CO$-dual $M^\vee$ induce a stable equivalence of Morita type.
The following observation is an immediate consequence of the fact that 
if $k\tenO A$ is selfinjective, then projective modules are relatively 
$\CO$-injective. For the convenience of the reader we give a direct 
proof (using some of the above comments).

\begin{Lemma} \label{selfinjLemma}
Suppose that $k\tenO A$ is selfinjective. Let $U$ be an $\CO$-free
$A$-module, let $Q$ be a finitely generated projective $A$-module,
and let $r$ be a positive integer. The map
$\Hom_A(U,Q) \to \Hom_A(U,Q/\pi^rQ)$
induced by the canonical surjection $Q\to$ $Q/\pi^rQ$ is surjective.
\end{Lemma}

\begin{proof}
Since $k\tenO A$ is selfinjective, it follows that every projective
indecomposable $k\tenO A$-module is the $k$-dual of a projective
indecomposable right $k\tenO A$-module. Since finitely generated
projective $k\tenO A$-modules lift uniquely, up to isomorphism, to
finitely generated projective $A$-modules, it follows that
$Q\cong$ $\Hom_\CO(R,\CO)$ for some finitely generated
projective right $A$-module $R$, and we have $Q/\pi^rQ\cong$ 
$\Hom_\CO(R,\CO/\pi^r\CO)$. The tensor-Hom adjunction yields
isomorphisms 
$$\Hom_A(U,Q)\cong\Hom_A(U,\Hom_\CO(R,\CO))\cong\Hom_\CO(R\tenA U,\CO)$$
and
$$\Hom_A(U,Q/\pi^rQ)\cong\Hom_A(U,\Hom_\CO(R,\CO/\pi^r\CO))\cong
\Hom_\CO(R\tenA U,\CO/\pi^r\CO)$$
Since $R$ is finitely generated projective as a right $A$-module and
$U$ is $\CO$-free, it follows that $R\tenA U$ is $\CO$-free. This
implies that the canonical map $\Hom_\CO(R\tenA U,\CO)\to$
$\Hom_\CO(R\tenA U,\CO/\pi^r\CO)$ is surjective, and hence so is
the canonical map $\Hom_A(U,Q)\to$ $\Hom_A(U,Q/\pi^rQ)$ as claimed.
\end{proof}

\section{On derivations and algebra automorphisms}
\label{derSection}

For background material on integrable derivations in the context of
deformations of algebras, see for instance Gerstenhaber \cite{Ger1} 
and Matsumura  \cite{Mat}. We first review some basic facts on 
integrable derivations, adapted to arbitrary  complete discrete 
valuation rings rather than rings of power series. 

\medskip
Let $A$ be an $\CO$-algebra which is free of finite rank as an 
$\CO$-module. We denote by $\Aut(A)$ the group of $\CO$-algebra 
automorphisms of $A$, by $\Inn(A)$ the normal subgroup of inner 
automorphisms, and by $\Out(A)=$ $\Aut(A)/\Inn(A)$ the group of outer 
$\CO$-algebra automorphisms of $A$. For any positive integer $r$ we 
denote by $\Aut_r(A)$ the subgroup of $\Aut(A)$ consisting of all 
$\CO$-algebra automorphisms of $A$  which induce the identity on 
$A/\pi^r A$, with the convention $\Aut_0(A)=$ $\Aut(A)$. That is,
$\Aut_r(A)$ is the kernel of the canonical map $\Aut(A)\to$
$\Aut(A/\pi^rA)$. We denote by $\Out_r(A)$ the image of $\Aut_r(A)$ in 
$\Out(A)$. 

\begin{Lemma} \label{OutrLemma}
Let $A$ be an $\CO$-algebra which is free of finite rank as an 
$\CO$-module, and let $r$ be a positive integer. 
The group $\Out_r(A)$ is equal to the kernel of the 
canonical map $\Out(A)\to$ $\Out(A/\pi^rA)$. 
\end{Lemma}

\begin{proof}
The image of 
$\Aut_r(A)$ in $\Out(A)$ belongs trivially to the kernel of the
canonical map $\Out(A)\to$ $\Out(A/\pi^rA)$. We need to show that
any element of this kernel is represented by an 
automorphism in $\Aut_r(A)$. Let $\alpha\in$ $\Aut(A)$ such that
$\alpha$ induces inner automorphism of $A/\pi^r A$, given by 
conjugation with an invertible element $\bar u$ in $A/\pi^r A$. Since 
the canonical map $A^\times\to$ $(A/\pi^rA)^\times$ is surjective, it 
follows that $\bar u$ lifts to an invertible element $u$ in $A$. 
The composition of $\alpha$ with conjugation by $u^{-1}$ yields an 
automorphism $\alpha'$ which belongs to the same class as $\alpha$ in 
$\Out(A)$ and which induces the identity on $A/\pi^rA$, hence which
belongs to $\Aut_r(A)$. This shows that every class in the
kernel of $\Out(A)\to$ $\Out(A/\pi^rA)$ has a representative
in $\Aut_r(A)$, as required.
\end{proof}

As mentioned before, the extra hypothesis on $Z(A)\to$ $Z(A/\pi^r A)$
being surjective in many of the statements below is satisfied 
automatically if $A=$ $\bar A\tenk k[[t]]$ for some $k$-algebra 
$\bar A$. 

\begin{Proposition} \label{automHH1-1}
Let $A$ be an $\CO$-algebra which is free of finite rank as an
$\CO$-module.
Let $r$ be a positive integer. Suppose that the canonical map $Z(A)\to$ 
$Z(A/\pi^r A)$ is surjective. Let $\alpha\in$ $\Aut_r(A)$, and let
$\mu : A\to$ $A$  be the unique linear map satisfying
$\alpha(a)=$ $a + \pi^r\mu(a)$ for all $a\in$ $A$. 

\smallskip\noindent (i)
The map $\bar\mu : A/\pi^rA\to$ $A/\pi^rA$ induced by $\mu$ is a 
derivation.  The class of $\bar\mu$ in $HH^1(A/\pi^rA)$ depends only on 
the class of $\alpha$ in $\Out(A)$.

\smallskip\noindent (ii)
If $\alpha$ is an inner automorphism of $A$, then $\alpha$ 
is induced by conjugation with an element of the form $c=$ $1+\pi^r d$ 
for some $d\in$ $A$, and we have $\bar\mu=$ $[\bar d,-]$; in particular, 
$\bar\mu$ is an inner derivation in that case.
\end{Proposition}

\begin{proof}
For $a$, $b\in$ $A$, comparing the expressions $\alpha(ab)=$ 
$ab+\pi^r\mu(ab)$ and $\alpha(a)\alpha(b)=$ 
$ab+\pi^r a\mu(b) + \pi^r\mu(a)b+\pi^{2r}\mu(a)\mu(b)$
yields the equality
$$\mu(ab) = a\mu(b) + \mu(a) b + \pi^r\mu(a)\mu(b)\ .$$
Reducing this modulo $\pi^r A$ shows that $\bar\mu$ is a derivation on 
$A/\pi^r A$. In order to show that the class of $\bar\mu$  in $HH^1(A)$ 
depends only on the image of $\alpha$ in $\Out(A)$ it suffices to prove 
(ii). Suppose that $\alpha$ is inner, induced by conjugation with an 
element $c\in$ $A^\times$. Since $\alpha$ induces the identity on 
$A/\pi^r A$, it follows that $\bar c\in$ $Z(A/\pi^rA)^\times$. 
Since the map $Z(A)\to$ $Z(A/\pi^r A)$ is assumed to be surjective, 
there is an element $z\in$ $Z(A)^\times$ such that $\bar z=$ $\bar c$,
hence such that $cz^{-1}\in$ $1+\pi^r A$. Thus, after replacing $c$ by 
$cz^{-1}$, we may assume that $c=$ $1+\pi^r d$ for some $d\in$ $A$. For 
$a\in$ $A$ we have $cac^{-1}=$ $\alpha(a)=$ $a + \pi^r\mu(a)$, hence 
$ca=$ $ac+\pi^r\mu(a)c$, or equivalently, $[c,a]=$ $\pi^r\mu(a)c$. 
Replacing $c$ by $1+\pi^r d$ in this equation and dividing by $\pi^r$ 
yields $[d,a]=$ $\mu(a)+\pi^r \mu(a)d$, whence $[\bar d,\bar a]=$
$\bar\mu(\bar a)$ as claimed in (ii). 
\end{proof}

\begin{Definition} \label{integrableDef}
Let $A$ be an $\CO$-algebra which is free of finite rank as an
$\CO$-module, and let $r$ be a positive integer. A derivation
$\delta$ on $A/\pi^r A$ is called {\it $A$-integrable} if 
there is an algebra automorphism $\alpha$ of $A$ and an
$\CO$-linear endomorphism $\mu$ of $A$ such that
$\alpha(a)=$ $a+\pi^r \mu(a)$ for all $a\in$ $A$ and such that
$\delta$ is equal to the map induced by $\mu$ on $A/\pi^r A$.
A class $\eta\in$ $HH^1(A/\pi^r A)$ is called {\it $A$-integrable} if
it can be represented by an $A$-integrable derivation. We denote 
by $HH^1_A(A/\pi^r A)$ the set of $A$-integrable classes in 
$HH^1(A/\pi^r A)$. 
\end{Definition}

\begin{Remark} \label{HasseSchmidtRemark}
One can extend the notation from Definition \ref{integrableDef}
in a way which is analogous to Hasse-Schmidt derivations of positive 
length considered for commutative algebras in \cite{NaMa}.
Let $s\geq 2r$ and let $\alpha$ be an automorphism of $A/\pi^sA$ 
which induces the identity on $A/\pi^r A$. Then, for
$\bar a=$ $a+\pi^sA\in$ $A/\pi^sA$ we have
$\alpha(\bar a)=\bar a + \pi^r\nu(\bar a)$ for some
element $\nu(a)$ of $A/\pi^sA$. Unlike in the preceding definition,
multiplication by $\pi^r$ on $A/\pi^sA$ is not injective, and hence
the map $\nu$ need not be a linear endomorphism of $A/\pi^sA$.
The annihilator of $\pi^r$ in $A/\pi^sA$ is $\pi^{s-r}A/\pi^sA$. 
Since $r\leq s-r$ it follows that $\nu$ induces a linear 
endomorphism of $A/\pi^rA$, and as before, this is a derivation. The
derivations and corresponding classes in $HH^1(A/\pi^rA)$ which arise
is this way are called {\it $A/\pi^sA$-integrable}. We denote
by $HH^1_{A/\pi^sA}(A/\pi^rA)$ the set of $A/\pi^sA$-integrable
classes. We have $HH^1_{A/\pi^{2r}A}(A/\pi^rA)=$
$HH^1(A/\pi^rA)$ because if $\mu$ is a linear endomorphism of $A$
inducing a derivation on $A/\pi^rA$, then $\alpha$ defined by
$\alpha(a+\pi^{2r}A) = a + \pi^r\mu(a) + \pi^{2r}A$ for any $a\in$ $A$
is an automorphism of $A/\pi^{2r}A$.
We have obvious inclusions $HH^1_A(A/\pi^rA)\subseteq$
$HH^1_{A/\pi^{s+1}A}(A/\pi^rA)\subseteq$ $HH^1_{A/\pi^sA}(A/\pi^rA)$.
These inclusions need not be equalities since an automorphism of 
$A/\pi^{s}A$ which induces the  identity on $A/\pi^rA$ does not 
necessarily lift to an automorphism of $A/\pi^{s+1}A$ or of $A$. 
Another variation of this definition arises from replacing $\pi$ by a 
suitable element in $Z(A)$ which is not a zero divisor. This point of 
view appears in Roggenkamp and Scott \cite[\S 5]{RoSc}. 
\end{Remark}

\begin{Proposition} \label{automHH1-2}
Let $A$ be an $\CO$-algebra which is free of finite rank as an 
$\CO$-module. Let $r$ be a positive integer. Suppose that the canonical 
map $Z(A)\to$ $Z(A/\pi^r A)$ is surjective. Let $\alpha\in$ $\Aut_r(A)$, 
and let $\mu : A\to$ $A$  be the unique linear map satisfying 
$\alpha(a)=$ $a + \pi^r\mu(a)$ for all $a\in$ $A$. Denote by $\bar\mu$ 
the derivation on $A/\pi^rA$ induced by $\mu$.

\smallskip\noindent (i)
The image of $\bar\mu$ in $HH^1(A/\pi^r A)$ is zero if and only if 
$\alpha$ induces an inner automorphism on $A/\pi^{2r} A$.

\smallskip\noindent (ii)
The map sending $\alpha$ to the class of $\mu$ induces a group 
homomorphism $\Out_r(A)\to$ $HH^1(A/\pi^rA)$, with kernel equal to
$\Out_{2r}(A)$. In particular, $HH^1_A(A/\pi^rA)$ is a subgroup of 
$HH^1(A/\pi^rA)$, and we have a short exact sequence of groups
$$\xymatrix{ 1\ar[r] & \Out_{2r}(A)\ar[r]&\Out_r(A)\ar[r]
& HH^1_A(A/\pi^r A)\ar[r]& 1}$$
\end{Proposition}

\begin{proof}
In order to show (i), suppose first that $\bar\mu=$ $[\bar d,-]$ for 
some $d\in$ $A$. Set $c=$ $1+\pi^r d$. Write $[d,a]=$ 
$\mu(a)-\pi^r\tau(a)$ for some $\tau(a)\in$ $A$. Then $\pi^r\mu(a)=$ 
$[c,a]+\pi^{2r}\tau(a) = ca-ac+\pi^{2r}\tau(a)$. Multiplying this
by $c^{-1}$ on the right implies that $cac^{-1}=$
$a + \pi^r\mu(a)c^{-1}-\pi^{2r}\tau(a)c^{-1}$. Using $\alpha(a)=$ 
$a+\pi^r\mu(a)$, it follows that $\alpha(a)-cac^{-1}=$ 
$\pi^r\mu(a)(1-c^{-1})+\pi^{2r}\tau(a)c^{-1}$. Since $c\in$ $1+\pi^r A$, 
we have $c^{-1}\in$ $1+\pi^r A$, hence $1-c^{-1}\in$ $\pi^r A$. This 
shows that $\alpha(a)-cac^{-1}\in$ $\pi^{2r} A$, hence
$\alpha$ induces an inner automorphism on $A/\pi^{2r}A$.
Suppose conversely that $\alpha$ acts as inner automorphism on
$A/\pi^{2r}A$. That is, the class of $\alpha$ belongs to the
kernel of the canonical map $\Out(A)\to$ $\Out(A/\pi^{2r}A)$.
By \ref{OutrLemma}, the class of $\alpha$ has a representative in 
$\Aut_{2r}(A)$.
Thus, after replacing $\alpha$ by a suitable representative in the
class of $\alpha$ in $\Out(A)$ we may assume that $\alpha\in$ 
$\Aut_{2r}(A)$. Thus $\alpha(a)=$ $a + \pi^{2r}\mu'(a)$ for some 
$\mu'(a)\in$ $A$, or equivalently, $\mu(a)=$ $\pi^r\mu'(a)$. This 
shows that $\mu$ induces the zero map on $A/\pi^r A$, proving (i). Let 
$\beta\in$ $\Aut_r(A)$ and $\nu$ such that $\beta(a)=$ $a+\pi^r\nu(a)$ 
for all $a\in$ $A$. A short calculation shows that $\beta(\alpha(a))=$
$a + \pi^r\mu(a)+\pi^r\nu(a)+\pi^{2r}\nu(\mu(a))=$
$\pi^r(\mu(a)+\nu(a)+\pi^r\nu(\mu(a)))$. This shows that the class 
determined by $\beta\circ\alpha$ in $HH^1(A/\pi^rA)$ is the class 
determined by $\bar\mu+\bar\nu$. Statement (ii) follows from (i). 
\end{proof}

The propositions above imply that $HH^1_A(A/\pi^rA)$ is a subgroup of 
$HH^1(A/\pi^rA)$, and that any inner derivation of $A/\pi^r A$ is 
$A$-integrable. Thus if a class $\eta\in$ $HH^1(A/\pi^rA)$ is 
$A$-integrable, then any derivation representing $\eta$ is 
$A$-integrable.
The group $HH^1_A(A/\pi^r A)$ depends in general on $A$, not just on 
$A/\pi^r A$; see the examples in \S \ref{Examplessection} below. If $A$ 
is clear from the context, we simply say that a class $\eta$, or a 
derivation representing $\eta$, is integrable instead of $A$-integrable. 
The arguments in the proof of the previous proposition extend in the
obvious way to $A/\pi^{2r}A$-integrable derivations, and yield the
following statement. The set $HH^1_{A/\pi^{2r}A}(A/\pi^rA)$ is a subgroup
of $HH^1(A/\pi^rA)$. Denote by $\Out_r(A/\pi^{2r}A)$ the group 
modulo inner automorphisms of algebra automorphisms of $A/\pi^{2r}A$ 
which induce the identity on $A/\pi^rA$.
Proposition \ref{automHH1-2} (ii) and the fact that
the image of $\Out_{2r}(A)$ in $\Out(A/\pi^{2r}A)$ is trivial imply
that we have a group isomorphism
$$\Out_r(A/\pi^{2r}A) \cong HH^1_{A/\pi^{2r}A}(A/\pi^rA)\ .$$

\section{Integrable derivations and Bockstein homomorphisms}

We interpret integrable derivations as images under certain Bockstein 
homomorphisms. Let $A$ be an $\CO$-algebra. Recall that if
$$\xymatrix{0\ar[r] & X\ar[r]^{\tau}& Y \ar[r]^{\sigma}& Z\ar[r] &0}$$
is a short exact sequence of cochain complexes of $A$-modules $X$, $Y$, 
$Z$ with differentials $\delta$, $\epsilon$, $\zeta$, respectively, 
and if $n$ is an integer, then the connecting homomorphism $H^n(Z)\to$ 
$H^{n+1}(X)$ is constructed as follows. Let $\bar z=$ 
$z+\Im(\zeta^{n-1})\in$ 
$H^n(Z)$ for some $z\in$ $\ker(\zeta^n)\subseteq$ $Z^n$. Let $y\in$ 
$Y^n$ such that $\sigma^n(y)=$ $z$. Then $\epsilon^n(y)\in$ $Y^{n+1}$
satisfies $\sigma^{n+1}(\epsilon^n(y))=$ $\zeta^{n}(\sigma^n(y))=$
$\zeta^{n}(z)=0$, hence $y\in$ $\ker(\sigma^{n+1})=$ $\Im(\tau^{n+1})$.
Thus there is $x\in$ $X^{n+1}$ such that $\tau^{n+1}(x)=$ 
$\epsilon^n(y)$. An easy verification shows that $x\in$ 
$\ker(\delta^{n+1})$ and that the class $\bar x=$ $x+\Im(\delta^n)\in$ 
$H^{n+1}(X)$ depends only on the class $\bar z$ of $z$ in $H^n(Z)$. The 
connecting homomorphism sends $\bar z$ to $\bar x$ as defined above. 
We use algebra automorphisms as subscripts to modules and bimodules as 
described in \ref{aut-notation} above. Write $A^e=$ $A\tenO A^\op$.
If $A$ is free of finite rank as an $\CO$-module, and if $\alpha\in$ 
$\Aut_r(A)$ for some positive integer $r$, then 
${_\alpha(A/\pi^r A)}\cong$ $A/\pi^r A\cong$ $(A/\pi^r A)_\alpha$ as 
$A^e$-modules, but the $A^e$-modules $A/\pi^{2r} A$ and 
$(A/\pi^{2r} A)_\alpha$ need not be isomorphic.

\begin{Proposition} \label{automHH1-3}
Let $A$ be an $\CO$-algebra which is free of finite rank as an
$\CO$-module. Let $r$ be a positive integer. Suppose that the canonical 
map $Z(A)\to$ $Z(A/\pi^r A)$ is surjective. Let $\alpha\in$ $\Aut_r(A)$, 
and let $\mu : A\to$ $A$  be the unique linear map satisfying
$\alpha(a)=$ $a + \pi^r\mu(a)$ for all $a\in$ $A$. 
Let $P$ be a projective resolution of $A$ as an $A^e$-module. Applying 
the functor $\Hom_{A^e}(P,-)$ to the the canonical short exact 
sequence of $A^e$-modules
$$\xymatrix{0\ar[r]&A/\pi^rA\ar[r]&(A/\pi^{2r}A)_\alpha\ar[r]&A/\pi^rA
\ar[r] & 0}$$
yields a short exact sequence of cochain complexes of $\CO$-modules
$$\xymatrix{0\ar[r]&\Hom_{A^e}(P,A/\pi^rA)
\ar[r] & \Hom_{A^e}(P,(A/\pi^{2r}A)_\alpha)
\ar[r] & \Hom_{A^e}(P,A/\pi^rA)
\ar[r] & 0}$$
The first nontrivial connecting homomorphism of the induced long
exact sequence in cohomology can be canonically identified with a map
$$\xymatrix{\End_{A^e}(A/\pi^rA)\ar[rr] &  & HH^1(A/\pi^rA) }$$
The image of $\Id_{A/\pi^rA}$ under this map is the class
of the derivation on $A/\pi^rA$ induced by $\mu$.
\end{Proposition}

\begin{proof}
We choose for $P$ the canonical projective resolution of $A$, which in 
degree $n\geq$ $0$ is equal to $A^{\ten n+2}$, and which is zero in 
negative degree. The canonical quasi-isomorphism $P\to$ $A$ is given 
by the map $A\tenO A\to$ $A$ induced by multiplication in $A$. The last 
nonzero differential of $P$ is the map 
$\delta_1 : A^{\ten 3}\to$ $A^{\ten 2}$ sending $a\ten b\ten c$
to $ab\ten c- a\ten bc$, where $a$, $b$, $c\in$ $A$. Using a standard 
sign convention, the first nonzero differential of the cochain complex 
$\Hom_{A^e}(P,(A/\pi^{2r})_\alpha)$ is given by precomposing with 
$-\delta_1$. By the remarks from the beginning of the previous section, 
we have canonical identifications
$$H^0(\Hom_{A^e}(P,A/\pi^rA))=HH^0(A,A/\pi^rA)\cong HH^0(A/\pi^rA)
=\End_{A^e}(A/\pi^rA)$$
The element on the left side corresponding to the identity map 
$\Id_{A/\pi^rA}$ on the right side is represented by the 
$A^e$-homomorphism
$$\zeta : A\tenO A \to A/\pi^rA$$ 
given by $\zeta(a\ten b) = ab + \pi^rA$ for all $a$, $b\in$ $A$.
This map lifts to an $A^e$-homomorphism
$$\hat\zeta : A\tenO A \to (A/\pi^{2r}A)_\alpha$$
defined by $\hat\zeta(a\ten b) = a\alpha(b) + \pi^{2r}A$ for
all $a$, $b\in$ $A$, thanks to the fact that $\alpha$ induces the
identity on $A/\pi^rA$. By the construction of connecting homomorphisms 
in a long exact cohomology sequence of a short exact sequence of cochain 
complexes, as reviewed above, we need to apply to the $0$-cochain 
$\hat\zeta$ the first nonzero differential of the complex  
$\Hom_{A^e}(P,(A/\pi^{2r}A)_\alpha)$.
This differential is given by precomposing with $-\delta_1$, hence 
yields the $1$-coboundary 
$$-\hat\zeta\circ\delta_1 : A^{\ten 3}\to (A/\pi^{2r}A)_\alpha$$
For all $a$, $b$, $c\in$ $A$ we have
$$(-\hat\zeta\circ\delta_1)(a\ten b\ten c) =
-\hat\zeta(ab\ten c+a\ten bc)=-ab\alpha(c)+a\alpha(bc)=
a(\alpha(b)-b)\alpha(c)\ = \pi^r a \mu(b) \alpha(c)$$
Thus $-\hat\zeta\circ\delta_1$ is the image of the map
induced by $\mu$ under the homomorphism
$\Hom_{A^e}(P,A/\pi^rA) \to
\Hom_{A^e}(P,(A/\pi^{2r}A)_\alpha)$
from the short exact sequence in the statement. 
The interpretation of $HH^1(A/\pi^rA)$ in terms of derivations comes 
from restricting an $A^e$-homomorphism
$A\tenO A\tenO A\to$ $A/\pi^rA$ to $1\ten A \ten 1$.
This means that the 
element in $HH^1(A/\pi^r A)$ determined by $\mu$ is the image of 
$\Id_{A/\pi^r A}$ as claimed.
\end{proof}

We write as before $A^e=$ $A\tenO A^\op$ and use the analogous
notation for the quotients $A/\pi^rA$ of $A$. By well-known standard
facts relating $\Ext$ and extensions, the representations of an 
element in $HH^1(A)$ by a short exact sequence and by a $1$-cycle 
$\varphi : A^{\ten 3}\to$ $A$ are related via a commutative diagram 
with exact rows
$$\xymatrix{\cdots\ar[r]
&A\tenO A\tenO A\ar[r]^{\delta_1}\ar[d]_{\varphi} 
& A\tenO A\ar[r]^{\delta_0}\ar[d]^{\sigma} 
& A\ar[r]\ar@{=}[d] & 0\\
0\ar[r]&A\ar[r]_{\iota} & X\ar[r] & A\ar[r]& 0}$$
The derivation $\mu$ representing this element is obtained
from restricting $\varphi$ to $1\ten A\ten 1$. Using 
$\delta_1(1\ten a\ten 1)=$ $a\ten 1-1\ten a$ we get that 
$\sigma(a\ten 1-1\ten a)=$ $\iota(\mu(a))$. We have a decomposition $X=$ 
$\iota(A)\oplus \sigma(A\ten 1)$ as a direct sum of left $A$-modules, 
both of which are isomorphic to $A$ as left $A$-modules. The first 
copy, $\iota(A)$, is isomorphic to $A$ as an $A^e$-module, but the 
second is not a bimodule: the right action on $\sigma(A\ten 1)$ is 
determined for $a\in$ $A$, by
$$\sigma(1\ten 1)\cdot a= \sigma(1\ten a)=
\sigma(a\ten 1)+\sigma(1\ten a)-\sigma(a\ten 1)=
\sigma(a\ten 1)-\iota(\mu(a))$$
In other words, after identifying $X=$ $A\oplus A$ as left
$A$-modules, the first copy of $A$ is an $A$-$A$-subbimodule, and the 
right action on the second copy is given by $(0,1)a=$ $(-\mu(a),a)$.
These considersations remain true for arbitrary commutative base rings,
so long as the algebra is projective over its base ring (that is,
we can apply these considerations to the $\CO/\pi^r\CO$-algebras
$A/\pi^rA$, for $r$ any positive integer). The following
result interprets the class of an integrable derivation as the
image of an element in the appropriate $\Ext$-module determined
by the first short exact sequence in the previous proposition.

\begin{Proposition} \label{automHH1-4}
Let $A$ be an $\CO$-algebra which is free of finite rank as an
$\CO$-module. Let $r$ be a positive integer. Suppose that the canonical 
map $Z(A)\to$ $Z(A/\pi^r A)$ is surjective. Let $\alpha\in$ $\Aut_r(A)$, 
and let $\mu : A\to$ $A$  be the unique linear map satisfying
$\alpha(a)=$ $a + \pi^r\mu(a)$ for all $a\in$ $A$. 
The exact sequence of $(A/\pi^{2r}A)^e$-modules 
$$\xymatrix{0\ar[r]&A/\pi^rA\ar[r]&(A/\pi^{2r}A)_\alpha\ar[r]&A/\pi^rA
\ar[r] & 0}$$
determines an element $\eta(\alpha)$ in 
$\Ext^1_{(A/\pi^{2r}A)^e}(A/\pi^rA, A/\pi^rA)$. 
The canonical surjective homomorphism $A/\pi^{2r}A\to$ $A/\pi^rA$ induces
a homomorphism
$$\xymatrix{\Ext^1_{(A/\pi^{2r}A)^e}(A/\pi^rA, A/\pi^rA)\ar[rr]
& & \Ext^1_{(A/\pi^{2r}A)^e}(A/\pi^{2r}A, A/\pi^rA)\cong
HH^1(A/\pi^rA)}$$
which sends $\eta(\alpha)$ to the class of the derivation induced
by $-\mu$.
\end{Proposition}

\begin{proof}
We consider the following commutative diagram with exact rows.
$$\xymatrix{
0\ar[r]&A/\pi^rA\ar[r]\ar@{=}[d]&(A/\pi^{2r}A)_\alpha\ar[r]
&A/\pi^rA\ar[r] &0\\
0\ar[r]&A/\pi^rA\ar[r]\ar@{=}[d]& X\ar[u] \ar[d]\ar[r]
&A/\pi^{2r}A\ar[r]\ar[u]\ar[d] &0\\
0\ar[r]&A/\pi^rA\ar[r]& X/\pi^rX \ar[r]
&A/\pi^{r}A\ar[r] & 0 }$$
The first row in this diagram is the canonical exact sequence 
representing the element $\eta(\alpha)$ in 
$\Ext^1_{(A/\pi^{2r}A)^e}(A/\pi^rA, A/\pi^rA)$ as in the statement.
The second row is obtained from the first by taking for $X$ the
pullback of the two canonical maps from $(A/\pi^{2r}A)_\alpha$ and
$A/\pi^{2r}A$ to $A/\pi^rA$. Thus the second row represents
the image of $\eta(\alpha)$ in
$\Ext^1_{(A/\pi^{2r}A)^e}(A/\pi^{2r}A, A/\pi^rA)$ as in the
statement. The third row represents the image of $\eta(\alpha)$
in $HH^1(A/\pi^rA)$ under the isomorphism
$\Ext^1_{(A/\pi^{2r}A)^e}(A/\pi^{2r}A, A/\pi^rA)\cong$ $HH^1(A/\pi^rA)$
from \ref{HHnrs}, applied with $s=2r$ and $n=1$. We have
$$X = \{(u+\pi^{2r}A, v+\pi^{2r}A)\ |\ u,v\in A,\ u-v\in\pi^rA\}$$
and the maps from $X$ to $(A/\pi^{2r}A)_\alpha$ and to 
$A/\pi^{2r}A$ are induced by the canonical projections.
The map $A/\pi^rA\to$ $X$ in the diagram sends $a+\pi^rA$ to
$(\pi^ra+\pi^{2r}A,0+\pi^{2r}A)$. An easy verification shows that
$$\pi^rX = \{(u+\pi^{2r}A, u+\pi^{2r}A)\ |\ u\in\pi^r A\}$$ 
The vertical maps from the second to the third row are
the canonical surjections. 
Denote by $Y$ the image of the map $A/\pi^rA\to$ $X$ in the diagram.
Set $$Z=\{(u+\pi^{2r}A, u+\pi^{2r}A)\ |\ u\in A\}$$
Note that $\pi^rZ=$ $\pi^rX$. Thus
the images of $Y$ and $Z$ in $X/\pi^rX$ yield a direct sum decomposition
of $X/\pi^rX$ as a left $A/\pi^rA$-module of two copies
isomorphic to $A/\pi^rA$. This is the decomposition as described in the
paragraph before the statement: the first of these two copies is
a bimodule, but the second is not. Moreover, the right action by $A$
on the second copy detects the negative of the derivation corresponding
to the element of $HH^1(A/\pi^rA)$ determined by the third row of
the above diagram. This right action by $a\in$ $A$ 
on the image in $X/\pi^rX$ of the element $(u+\pi^{2r}A,u+\pi^{2r}A)\in$ 
$Z$ is represented by the element in $X$ of the form
$$(u\alpha(a)+\pi^{2r}A, ua+\pi^{2r}A) =
(ua+\pi^{2r}A,ua+\pi^{2r}A) + (\pi^ru\mu(a) +\pi^{2r}A, 0+\pi^{2r}A)$$
The image in $X/\pi^rX$ of the second element on the right side is the 
image of $(u\mu(a)+\pi^rA)$ under the map $A/\pi^rA\to$ $X/\pi^rX$. The 
result follows from the identification between derivations and short 
exact sequences representing elements in $HH^1(A/\pi^rA)$ as described in
the paragraph preceding this proposition.
\end{proof}

The fact that we obtain the class represented by $\mu$ in Proposition 
\ref{automHH1-3} and the class represented by $-\mu$ in Proposition 
\ref{automHH1-4} is essentially a matter of sign conventions. For the 
sake of completeness, we mention that Bockstein homomorphisms in 
Hochschild cohomology are graded derivations, just as in the case of 
group cohomology (cf. \cite[Lemma 4.3.3]{BenII}), for instance. We 
sketch a proof for the convenience of the reader. As before, we will 
use without further mention the canonical isomorphisms 
$HH^n(A;A/\pi^rA)\cong$ $HH^n(A/\pi^rA)$ for any integers $n\geq$ $0$, 
$r>0$, and any $\CO$-algebra $A$ which is free of finite rank as an 
$\CO$-module.

\begin{Proposition} \label{Bocksteinderivation}
Let $A$ be an $\CO$-algebra which is free of finite rank as an
$\CO$-module. Let $r$ be a positive integer, and set $\bar A=$ 
$A/\pi^r A$. Let $P$ be a projective resolution of $A$ as an 
$A^e$-module. For any nonnegative integer $n$ denote by 
$\beta_n : HH^n(\bar A)\to$ $HH^{n+1}(\bar A)$ the connecting 
homomorphism in the long exact cohomology sequence of the canonical 
short exact sequence of cochain complexes
$$\xymatrix{0\ar[r]&\Hom_{A^e}(P,\bar A)\ar[r] 
& \Hom_{A^e}(P,A/\pi^{2r} A)\ar[r] 
& \Hom_{A^e}(P,\bar A)\ar[r]& 0}$$ 
Then for any nonnegative integers $m$, $n$, any $\zeta\in$ 
$HH^m(\bar A)$ and any $\eta\in$ $HH^n(\bar A)$ we have
$$\beta_{m+n}(\zeta\eta) = \beta_m(\zeta)\eta + 
(-1)^m\zeta\beta_n(\eta)\ .$$
\end{Proposition}

\begin{proof}
Denote by $\epsilon_n : P_n\to$ $P_{n-1}$ the differential of $P$
for $n>0$, and by $\epsilon_0 : P_0\to$ $A$ a map with kernel
$\Im(\epsilon_1)$ inducing a quasi-isomorphism $P\to$ $A$.
We have $A\tenA A\cong$ $A$, and hence $P\tenA P$ is also a
projective resolution of $A$. Thus there is a homotopy
equivalence $h : P\simeq P\tenA P$ which lifts the identity
on $A$. The homotopy class of $h$ with this property is unique.
The term in degree $n\geq$ $0$ of $P\tenA P$ is the direct sum
of the terms $P_i\tenA P_j$, with $i$, $j$ running over all
nonnegative integers  such that $i+j=$ $n$. The differential
$\hat\epsilon$ of $P\tenA P$ is obtained by taking the sum of the maps
$(-1)^i\Id\ten\epsilon_j : P_i\tenA P_j\to$ $P_i\tenA P_{j-1}$  and
$\epsilon_i\ten\Id : P_i\tenA P_j\to$ $P_{i-1}\tenA P_j$.

Let $m$, $n$ be nonnegative integers, $\zeta\in$ $HH^m(\bar A)$ 
and $\eta\in$ $HH^n(\bar A)$. These classes are represented
by cocycles, abusively denoted by the same letters, 
$\zeta : P_m\to$ $\bar A$ and $\eta : P_n\to$ $\bar A$. 
The product in $HH^*(\bar A)$ is induced by the map sending
$\zeta$ and $\eta$ to the cocycle $(\zeta\ten\eta)\circ h_{m+n}$.
Via the canonical surjection $A/\pi^{2r}A\to$ $\bar A$, we lift the 
cocycles $\zeta$ and $\eta$ to cocycles $\hat\zeta : P_m\to$ 
$A/\pi^{2r}A$ and $\hat\eta : P_n\to$ $A/\pi^{2r}A$. 
The cocycle
$$\widehat{\zeta\eta} = (\hat\zeta\ten\hat\eta)\circ h_{m+n} :
P_{m+n} \to A/\pi^{2r}A\ $$
lifts then a cocycle $P_{m+n}\to $ $\bar A$ which represents
the product $\zeta\eta\in$ $HH^{m+n}(\bar A)$.

Using these cocycles, we calculate the images under the Bockstein 
homomorphisms of $\zeta$, $\eta$, and $\zeta\eta$.
The connecting homomorphism $\beta_m$ sends $\zeta$ to
the class of a cocycle $\tilde\zeta : P_{m+1}\to$ $\bar A$
which upon identifying $\bar A$ with its image $\pi^rA/\pi^{2r}A$ in
$A/\pi^{2r}$ is equal to 
$$\pi^r\hat\zeta\circ\epsilon_{m+1} : P_{m+1}\to \bar A\ .$$
Similarly, $\beta_n$ sends $\eta$ to the class of a cocycle
$\tilde\eta : P_{n+1}\to$ $\bar A$ equal to 
$\pi^r\hat\eta\circ\epsilon_{n+1}$, with the analogous identification.
The connecting homomorphism $\beta_{m+n}$ sends $\zeta\eta$
to a class in $HH^{m+n+1}(A)$ represented by
a cocycle $P_{m+n+1}\to$ $\bar A$  which upon identifying
$\bar A$ with its image $\pi^rA/\pi^{2r}A$ in $A/\pi^{2r}A$ is
equal to the cocycle 
$$\pi^r(\hat\zeta\ten\hat\eta)\circ\hat\epsilon_{m+n+1}
\circ h_{m+n+1}\ .$$
In view of the explicit description of the differential of $P\tenA P$
above, this expression is equal to
$$\pi^r(\hat\zeta\circ\epsilon_{m+1}\ten\hat\eta +
(-1)^m\hat\zeta\ten\hat\eta\circ\epsilon_{n+1})\circ h_{m+n+1} =
(\pi^r\hat\zeta\circ\epsilon_{m+1}\ten\hat\eta +
(-1)^m\hat\zeta\ten\pi^r\hat\eta\circ\epsilon_{n+1})
\circ h_{m+n+1}\ .$$
By the above description of the images of $\zeta$ and $\eta$
under Bockstein homomorphisms, this is equal to 
$$(\tilde\zeta \ten\hat\eta +
(-1)^m\hat\zeta\ten\tilde\eta)\circ h_{m+n+1}\ $$
This represents $\beta(\zeta)\eta + (-1)^m\zeta\beta(\eta)$
as required.
\end{proof}

\section{Integrable derivations and stable equivalences}

Let $A$ be an $\CO$-algebra which is free of finite rank as
an $\CO$-module. The group $\Out(A)$ need not be invariant under 
Morita equivalences, but the subgroups $\Out_r(A)$ are invariant 
under Morita equivalences. Indeed, if $\alpha\in$ $\Out_r(A)$ for 
some $r>0$, then $\alpha$ induces in particular the identity on 
$A/\pi A$, hence stabilises the isomorphism class of any simple
$A/\pi A$-module, hence of any projective indecomposable $A$-module.
The subgroup of $\Out(A)$ represented by those automorphisms
which stabilise all projective indecomposable modules is
well-known to be invariant under Morita equivalences, and thus so 
are the groups $\Out_r(A)$.
More explicitly, if $A$, $B$ are two $\CO$-free algebras
of finite $\CO$-rank and if $M$ is an $A$-$B$-bimodule
inducing a Morita equivalence, then for any $\alpha\in$
$\Aut_r(A)$ there is $\beta\in$ $\Aut_r(B)$ such that
${_{\alpha^{-1}}{M}}\cong$ $M_\beta$ as $A$-$B$-bimodules,
and the correspondence $\alpha\mapsto\beta$ yields a
group isomorphism $\Out_r(A)\cong$ $\Out_r(B)$. As before, the use of 
automorphisms as subscripts to modules and bimodules is as 
in \ref{aut-notation} above.
We use without further reference some standard facts relating 
bimodules and automorphisms (see e. g. \cite[\S 55A]{CR2}): if 
$\alpha\in$ $\Aut(A)$, then the $A$-$A$-bimodule ${{A}_\alpha}$ induces 
a Morita equivalence on $A$, we have $A_\alpha\cong$ $A$ as 
$A$-$A$-bimodules if and only if $\alpha$ is inner, the map sending 
$a\in$ $A$ to $\alpha(a)$ is an $A$-$A$-bimodule isomorphism 
${_{\alpha^{-1}}{A}} \cong$ ${{A}_\alpha}$, and the map sending
$\alpha$ to $A_\alpha$ induces an injective group homomophism
from $\Out(A)$ to the Picard group $\Pic(A)$ of isomorphism
classes of bimodules inducing Morita equivalences on $A$.
We will show next that if the involved algebras are relatively 
$\CO$-selfinjective with separable semisimple quotients, then the
groups $\Out_r(A)$ are in fact invariant under stable equivalences of 
Morita type in a way which is compatible with the associated
groups of integrable derivations. Symmetric $\CO$-algebras
remain symmetric over $k$, hence selfinjective, and thus
Theorem \ref{integrableinvariant} is a special case of the
following result.

\begin{Theorem} \label{outmstable}
Let $A$, $B$ be $\CO$-algebras which are free of finite rank
as $\CO$-modules, such that the $k$-algebras  $k\tenO A$
and $k\tenO B$ are indecomposable nonsimple selfinjective with 
separable semisimple quotients. Let $r$ be a positive integer.
Suppose that the canonical maps $Z(A)\to$ $Z(A/\pi^rA)$ and 
$Z(B)\to$ $Z(B/\pi^rB)$ are surjective. 
Let $M$ be an $A$-$B$-bimodule and $N$ a $B$-$A$-bimodule inducing 
a stable equivalence of Morita type between $A$ and $B$. 
For any $\alpha\in$ $\Aut_r(A)$ there is $\beta\in$ 
$\Aut_r(B)$ such that ${_{\alpha^{-1}}{M}}\cong$ $M_\beta$ as 
$A$-$B$-bimodules, and the correspondence $\alpha\mapsto$ $\beta$ 
induces a group isomorphism $\Out_r(A)\cong$ $\Out_r(B)$ making the 
following diagram of groups commutative:
$$\xymatrix{\Out_r(A) \ar[rr]^{\cong}\ar[d] && \Out_r(B) \ar[d] \\
HH^1_A(A/\pi^r A) \ar[rr]^{\cong} & & HH^1_B(B/\pi^r B) }\ $$
where the vertical maps are the group homomorphisms from
Proposition \ref{automHH1-2} (ii) and where the lower horizontal 
isomorphism is induced by the functor $N\tenA - \tenA M$.
\end{Theorem}

If $M$ and $N$ induce a Morita equivalence, then the hypothesis on
$k\tenO A$ and $k\tenO B$ being selfinjective with separable semisimple
quotients is not needed; see the Remark \ref{MoritaRemark} below for the 
necessary adjustments. 
This yields an alternative proof of a result due to  Farkas, Geiss, 
and Marcos in \cite{FGM}. The first part of the statement of 
\ref{outmstable} is the following variation of \cite[Theorem 4.2]{Listable}.

\begin{Lemma} \label{aMbLemma1}
With the notation and hypotheses of \ref{outmstable},
for any $\alpha\in$ $\Aut_r(A)$ there is $\beta\in$ 
$\Aut_r(B)$ such that ${_{\alpha^{-1}}{M}}\cong$ $M_\beta$ as 
$A$-$B$-bimodules, and the correspondence $\alpha\mapsto$ $\beta$ 
induces a group isomorphism $\Out_r(A)\cong$ $\Out_r(B)$.
\end{Lemma}

\begin{proof}
Let $\alpha\in$ $\Aut_r(A)$. Then $\alpha$ induces the identity
on $A/\pi A\cong$ $k\tenO A$, hence stabilises the isomorphism
classes of all $A/\pi A$-modules, hence in particular of all
simple $A/\pi A$-modules and of all finitely generated projective 
$A$-modules. By \cite[Theorem 4.2]{Listable} there exists an 
automorphism $\beta$ of $B$, unique up to inner automorphisms, such 
that we have an isomorphism of $A$-$B$-bimodules
${_{\alpha^{-1}}{M}}\cong$ $M_\beta$. We sketch the argument to prove 
this. The self-stable equivalence induced by 
$N\tenA {_{\alpha^{-1}}{M}}$ reduces to the identity functor on the 
stable module category of $k\tenO B$-modules because $\alpha$ induces 
the identity on $k\tenO A$. By \cite[Theorem 2.1]{Listable}, a stable 
equivalence of Morita type preserving simple modules is induced by a 
Morita equivalence. Since this Morita equivalence fixes the isomorphism 
class of any simple $B$-module, it is given by a bimodule of the form
$B_\beta$, for some $\beta\in$ $\Aut(B)$. Thus the $B$-$B$-bimodule
$B_\beta$ is, up to isomorphism, the unique nonprojective direct 
summand of $N\tenA {_{\alpha^{-1}}{M}}$. Tensoring by $M\tenB-$
shows that $M_\beta\cong$ ${_{\alpha^{-1}}{M}}$. 
We observe next that $\beta$ belongs to $\Out_r(B)$. We have an 
isomorphism of $A/\pi^rA$-$B/\pi^rB$-bimodules 
$$M_\beta/\pi^rM_\beta\cong 
{_{\alpha^{-1}}{M}}/\pi^r({_{\alpha^{-1}}{M}})\ .$$ 
Since $\alpha$ induces the identity on $A/\pi^r A$, it follows that
this bimodule is isomorphic to $M/\pi^r M$. Tensoring with
$N$ on the left implies that the nonprojective summands of
$N\tenA M/\pi^r(N\tenB M)$ and $N\tenA M_\beta/\pi^r(N\tenB M_\beta)$ 
are isomorphic, hence that $B/\pi^rB$ and $B_\beta/\pi^r B_\beta$ are
isomorphic as $B$-$B$-bimodules. This implies that the
automorphism on $B/\pi^rB$ induced by $\beta$ is inner.
A straightforward verification shows that the correspondence 
$\alpha\mapsto$ $\beta$ induces a group homomorphism $\Out_r(A)\to$ 
$\Out_r(B)$. Exchanging the roles of $A$ and $B$ yields the 
inverse of this group homomorphism. 
\end{proof}

It follows immediately from Lemma \ref{aMbLemma1}, applied
with $r$ and $2r$, together with Proposition \ref{automHH1-1}
and \ref{automHH1-2} that in the situation of Theorem \ref{outmstable} 
there is a group isomorphism $HH^1_A(A/\pi^r A)\cong$ $HH^1_B(B/\pi^r B)$
which makes the diagram in \ref{outmstable} commutative. In order 
to show that this isomorphism is induced by the functor 
$N\tenA - \tenA M$, we will need the following technical 
observation.

\begin{Lemma} \label{aMbLemma2}
With the notation and hypotheses of \ref{outmstable}, 
let $\alpha\in$ $\Aut_r(A)$ and $\beta\in$ 
$\Aut_r(B)$ such that ${_{\alpha^{-1}}{M}}\cong$ $M_\beta$
as $A$-$B$-bimodules. Then there is an $A$-$B$-bimodule isomorphism 
$\varphi : {_{\alpha^{-1}}{M}}\cong$ $M_\beta$
which induces the identity map on $M/\pi^rM$.
\end{Lemma}

\begin{proof}
Let $\psi : {_{\alpha^{-1}}{M}}\cong$ $M_\beta$ be an isomorphism
of $A$-$B$-bimodules. That is, $\psi$ is an $\CO$-linear automorphism
of $M$ satisfying
$$\psi(\alpha^{-1}(a)mb) = a\psi(m)\beta(b)$$
for all $a\in$ $A$, $b\in$ $B$, and $m\in$ $M$. 
Set $\bar\CO=$ $\CO/\pi^r\CO$. Then $\bar A=$ $A/\pi^r A$ and
$\bar B=$ $B/\pi^r B$ are $\bar\CO$-algebras, and the
$\bar A$-$\bar B$-bimodule $\bar M=$ $M/\pi^r M$ and the
$\bar B$-$\bar A$-bimodule $\bar N=$ $N/\pi^r N$ induce a stable
equivalence of Morita type between $\bar A$ and $\bar B$.
Multiplication by elements in $Z(\bar A)$ induces an
$\bar\CO$-algebra isomorphism $Z(\bar A)\cong$ 
$\End_{\bar A\ten_{\bar\CO} \bar A^\op}(\bar A)$. We denote by 
$Z^\pr(\bar A)$ the ideal in $Z(\bar A)$ corresponding to those
endomorphisms of $\bar A$ which factor through a projective 
$\bar A\ten_{\bar\CO} \bar A^\op$-module. Following \cite[\S 5]{BroueEq}
we set $Z_\st(\bar A)=$ $Z(\bar A)/Z^\pr(\bar A)$. We use the analogous
notation for $\bar B$. We denote by 
$\Endbar_{\bar A\ten_{\bar\CO}\bar B^\op}(\bar M)$ the quotient of
$\End_{\bar A\ten_{\bar\CO}\bar B^\op}(\bar M)$ by the ideal of endomorphisms
which factor through a projective $\bar A\ten_{\bar\CO}\bar B^\op$-module.
By the proof of \cite[5.4]{BroueEq} we have $\bar\CO$-algebra isomorphisms
$$Z_\st(\bar A)\cong\Endbar_{\bar A\ten_{\bar\CO}\bar B^\op}(\bar M)
\cong Z_\st(\bar B)$$
induced by left and right mulitplication with elements of $Z(\bar A)$
and $Z(\bar B)$ on $\bar M$, respectively. The map $\bar\psi$ on
$\bar M$ induced by $\psi$ is in fact an endormorphism of $\bar M$
as an $\bar A$-$\bar B$-bimodule, because $\alpha$ and $\beta$
induce the identity on $\bar A$ and $\bar B$, respectively. Since
the map $Z(A)\to$ $Z(\bar A)$ is assumed to be surjective, there
is $z\in$ $Z(A)$ such that the $A$-$B$-endomorphism $\lambda$ of $M_\beta$
given by left multiplication with $z$ induces an
$\bar A$-$\bar B$-endomorphism $\bar\lambda$ on $\bar M$ belonging to
the same class as that of $\bar\psi$ in 
$\Endbar_{\bar A\ten_{\bar\CO}\bar B^\op}(\bar M)$.
Thus, after replacing $\psi$ by $z^{-1}\cdot \psi$, we may assume that
the $\bar A$-$\bar B$-endomorphism $\bar\psi-\Id_{\bar M}$ of $\bar M$
factors through a projective cover $\bar Q$ of $\bar M$; say
$\bar\psi-\Id_{\bar M} =$ $\bar\epsilon\circ\bar\delta$ for some
bimodule homomorphisms $\bar\delta : \bar M\to$ $\bar Q$ and
$\bar\epsilon : \bar Q\to$ $\bar M$. Let $Q$ be a projective cover
of $M_\beta$. Since $\beta$ induces the identity on $\bar B$, it follows 
that $Q/\pi^r Q$ is a projective cover of $\bar M$, hence isomorphic
to $\bar Q$. Choose notation such that $Q/\pi^r Q=$ $\bar Q$. 
The bimodule $\bar\epsilon$ lifts to an $A$-$B$-bimodule homomorphism  
$\epsilon : Q\to$ $M_\beta$ because $Q$ is projective.
By the assumptions, the algebra $k\tenO (A\tenO B\op)$ is selfinjective.
Since the involved modules $M$ and $Q$ are $\CO$-free, it follows
from Lemma \ref{selfinjLemma} that the bimodule homomorphisms 
$\bar\delta$ lifts to an $A$-$B$-bimodule homomorphism 
$\delta : {_{\alpha^{-1}}{M}}\to$ $Q$.
Since $M$ is indecomposable nonprojective,
it follows that $\psi-\epsilon\circ\delta : {_{\alpha^{-1}}{M}}\cong$ 
$M_\beta$ is still an isomorphism. By construction, 
$\psi-\epsilon\circ\delta$ induces the identity on $\bar M$,
whence the result.
\end{proof}

\begin{proof}[Proof of Theorem \ref{outmstable}]
The first statement holds by Lemma \ref{aMbLemma1}. 
It remains to verify the commutativity of the diagram in the statement.
Let $\alpha\in$ $\Aut_r(A)$ and $\beta\in$ $\Aut_r(B)$ such
that ${_{\alpha^{-1}}{M}}\cong$ $M_\beta$. By Lemma \ref{aMbLemma2}, there
is an isomorphism $\varphi : {_{\alpha^{-1}}{M}}\cong$ $M_\beta$ which 
induces the identity on $M/\pi^rM$. 
Let $\mu$ and $\nu$ be the $\CO$-linear endomorphisms of
$A$ and $B$, respectively, satisfying $\alpha(a)=$ $a+\pi^r\mu(a)$
for all $a\in$ $A$ and $\beta(b)=$ $b+\pi^r\nu(b)$ for all
$b\in$ $B$. Set $\bar A=$ $A/\pi^rA$ and $\bar B=$ $B/\pi^rB$.
Denote by $\bar\mu$ and $\bar\nu$ the classes 
in $HH^1_A(\bar A)$ and $HH^1_B(\bar B)$ determined by $\mu$ and $\nu$;
that is, $\bar\mu$ and $\bar\nu$ are the images of the classes of 
$\alpha$ and $\beta$ under the canonical group homomorphisms 
$\Out_r(A)\to$ $HH^1(\bar A)$ and $\Out_r(B)\to$ $HH^1(\bar B)$,
respectively. Set $\bar M=$ $M/\pi^rM$ and $\bar\CO=$ $\CO/\pi^r\CO$.
The assumptions imply that tensoring by $\bar M$ yields a
stable equivalence of Morita type between $\bar A$ and $\bar B$.
In particular, we have isomorphisms
$$HH^1(\bar A) \cong 
\Ext^1_{\bar A\ten_{\bar\CO}\bar B^\op}(\bar M,\bar M)
\cong HH^1(\bar B)$$
induced by the functors $-\tenA M$ and 
$M\tenB-$. Since $N\tenA M$ is isomorphic to $B$
in the relatively $\CO$-stable category of $B\tenO B^\op$-modules,
it follows that the composition 
$$HH^1(\bar A) \cong HH^1(\bar B)$$
of the two previous isomorphisms is induced by the functor
$N\tenA - \tenA M$. Thus we need to show that $M\tenB-$ and $-\tenA M$
send $\bar\nu$ and $\bar\mu$, respectively, to the the same class in
$\Ext^1_{\bar A\ten_{\bar\CO}\bar B^\op}(\bar M,\bar M)$.
Note that the functors $M\tenB-$ and $-\tenA M$ induce obvious algebra 
homomorphisms
$$\xymatrix{\End_{\bar A^e}(\bar A)\ar[r] &
\End_{\bar A\ten_{\bar\CO}\bar B^\op}(\bar M) & 
\End_{\bar B^e}(\bar B) \ar[l] }\ ,$$
where $\bar A^e=$ $\bar A\ten_{\bar\CO}\bar A^\op$ and
$\bar B^e=$ $\bar B\ten_{\bar\CO}\bar B^\op$.
Tensoring the short exact sequences
$$\xymatrix{0\ar[r]&\bar A\ar[r]&(A/\pi^{2r}A)_\alpha\ar[r]&\bar A
\ar[r] & 0}$$
$$\xymatrix{0\ar[r]&\bar B\ar[r]&(B/\pi^{2r}B)_\beta\ar[r]&\bar B
\ar[r] & 0}$$
by $-\tenA M$ and $M\ten_{B}-$, respectively, yields short 
exact sequences of the form
$$\xymatrix{0\ar[r]&\bar M\ar[r]&{_{\alpha^{-1}}{(M/\pi^{2r}M)}}
\ar[r]&\bar M \ar[r] & 0}$$
$$\xymatrix{0\ar[r]&\bar M\ar[r]&(M/\pi^{2r}M)_\beta
\ar[r]&\bar M \ar[r] & 0}$$
These two exact sequences are equivalent, because the isomorphism
$\varphi$ above induces an isomorphism
${_{\alpha^{-1}}{(M/\pi^{2r}M)}}\cong$ $(M/\pi^{2r}M)_\beta$
which induces in turn the identity map on $\bar M$.
Applying  the functor $\Hom_{A\tenO B^\op}(M,-)$ to these
two short exact sequences, with identifications analogous to
those in Proposition \ref{automHH1-3}, yields therefore the same 
connecting homomorphism
$$\End_{\bar A\ten_{\bar\CO}\bar B^\op}(M)\to
\Ext^1_{\bar A\ten_{\bar\CO}\bar B^\op}(\bar M, \bar M)$$
Denote by $\bar\mu\ten\Id_{\bar M}$ and $\Id_{\bar M}\ten\bar\nu$
the images in $\Ext^1_{\bar A\ten_{\bar\CO}\bar B^\op}(\bar M, \bar M)$
of $\bar\mu$ and $\bar\nu$ under the functors $-\tenA M$ and $M\tenB-$,
respectively. By the naturality properties of connecting homomorphisms, 
it follows from the description of $\bar\mu$ and $\bar\nu$ 
in Proposition \ref{automHH1-3}, that $\bar\mu\ten\Id_{\bar M}$ and 
$\Id_{\bar M}\ten\bar\nu$ are both equal to the image of $\Id_{\bar M}$
under the above connecting homomorphism. This shows that the group 
isomorphism $HH^1_A(\bar A)\cong$ $HH^1_B(\bar B)$ induced by the group 
isomorphism $\Out_r(A)\cong$ $\Out_r(B)$ is equal to the group 
isomorphism determined by the functor $N\tenA - \tenA M$, completing the 
proof of Theorem \ref{outmstable}.
\end{proof}

\begin{Remark} \label{MoritaRemark}
With the notation of Theorem \ref{outmstable}, if $M$ and $N$ induce a 
Morita equivalence, then the conclusions of the theorem and the lemmas 
in this section hold without the assumption that $A$ and $B$ are 
relatively $\CO$-injective with separable semisimple quotients. This 
hypothesis is needed in Lemma \ref{aMbLemma1} via there reference 
\cite[Theorem 4.2]{Listable}, but it is not needed if $M$ and $N$ 
induce a Morita equivalence, by the remarks at the beginning of this 
section. This hypothesis is also needed in the proof of Lemma 
\ref{aMbLemma2}, but again, if $M$ and $N$ induce a Morita equivalence, 
then left and right multiplication by elements in $Z(\bar A)$ and 
$Z(\bar B)$, respectively, induce algebra isomorphisms
$$Z(\bar A)\cong  \End_{\bar A\ten_{\bar\CO} \bar B^\op}(\bar M)
\cong Z(\bar B)\ ,$$
where the notation is as in the proof of \ref{aMbLemma2}. Thus 
$\bar\psi$ is induced by multiplication with an element $\bar z\in$ 
$Z(\bar A)$, hence after replacing $\psi$ by $z^{-1}\cdot\psi$ we get 
that $\bar\psi=$ $\Id_{\bar M}$. As mentioned above, this shows the 
invariance of the groups of integrable derivations under Morita 
equivalences, a result which is due to  Farkas, Geiss, and Marcos 
\cite{FGM} in the case $\CO=$ $k[[t]]$, here slightly generalised to 
arbitrary complete discrete valuation rings.
\end{Remark}

\section{Examples} \label{Examplessection}

The following examples illustrate some of the basic connections between 
integrable elements in $HH^1(A/\pi^mA)$, the algebra $A$, and the 
ramification of the ring $\CO$.

\begin{Example} \label{abelianpgroupExample}
Suppose that $k$ has prime characteristic $p$ and that $\CO$
has characteristic zero.  Let $P$ be a nontrivial finite abelian 
$p$-group. Suppose that $K$ contains a primitive $|P|$-th root of 
unity $\tau$. Let $r$ be the positive integer such that $\pi^r\CO=$ 
$(\tau-1)\CO$.

\smallskip\noindent {\bf (a)}
Any $\CO$-algebra automorphism of $\OP$ which induces the identity
on $kP$ is of the form $\alpha(u)=$ $\zeta(u)u$ for all $u\in$ $P$, 
where $\zeta\in$ $\Hom(P,\CO^\times)$. Thus $\Aut_1(\OP)\cong$ 
$\Hom(P,\CO^\times)$, which is (noncanonically) isomorphic to $P$, and 
the elements in $HH^1_{\OP}(kP)$ are induced by linear endomorphisms 
of $\OP$ sending $u\in$ $P$ to $\frac{\zeta(u)-1}{\pi}u$ for all $u\in$ 
$P$. If $r\geq$ $2$, then $\pi$ divides the coefficients 
$\frac{\zeta(u)-1}{\pi}$, and hence the induced linear endomorphisms of 
$kP$ are zero. Thus if $r\geq$ $2$, then $HH^1_\OP(kP)=$ $\{0\}$.

\smallskip\noindent {\bf (b)}
Set $\bar\CO=$ $\CO/(\tau-1)\CO=$ $\CO/\pi^r\CO$. Any $\CO$-algebra 
automorphism of $\OP$ which induces the identity on $kP$ induces by the 
above description the identity on $\bar\CO P$, hence is as before of the 
form $\alpha(u)=$ $\zeta(u)u$ for all $u\in$ $P$, where $\zeta : P\to$ 
$\CO^\times$ is a group homomorphism. Thus, as before, we have 
$\Aut_r(\OP)\cong$ $\Hom(P,\CO^\times)$. However, 
$HH^1_{\OP}(\bar\CO P)$ is a nontrivial finite subgroup of 
$HH^1(\bar\CO P)$ consisting of all derivations $d_\zeta$ of $\bar\CO P$ 
induced by the linear maps on $\OP$ sending $u\in$ $P$ to 
$\frac{\zeta(u)-1}{\tau-1}u$, where $\zeta\in$ $\Hom(P,\CO^\times)$.

\smallskip\noindent {\bf (c)}
It is shown in \cite{FGM} that if $\CO=$ $k[[t]]$, then any integrable
derivation on $A/tA$ preserves the Jacobson radical of $A/tA$. The 
result does not carry over to general $\CO$ and the Artinian quotients 
$A/\pi^r A$ of $A$. Choose $P$ cyclic of order $p$ with a generator 
denoted $y$. The automorphism $\alpha$ of $\OP$ sending $y$ to $\tau y$ 
determines a derivation of $\bar\CO P$ induced by the linear 
endomorphism $\mu$ of $\OP$ sending $y^i$ to 
$\frac{\tau^i-1}{\tau-1} y^i$. This endomorphism sends $y-1$ to $y$, 
and hence the derivation on $\bar\CO P$ induced by $\mu$ does not 
preserve the Jacobson radical.
\end{Example}

\begin{Example}
Different lifts of a symmetric $k$-algebra yield in general different 
groups of integrable derivations. Suppose that $k$ has prime 
characteristic $p$. Let $P$ be a finite cyclic $p$-group of order $p^s$ 
for some positive integer $s$; let $y$ be a generator of $P$. 
There is a $k$-algebra isomorphism $kP\cong$ $k[x]/(x^{p^s})$ sending 
$y-1$ to $x+ (x^{p^s})$. We have $\dim_k(HH^1(kP))=$ $p^s$; more precisely,
for any $i$ such that $0\leq i\leq p^s-1$ there is a unique derivation
$d_i$ on $kP$ sending $x$ to $x^i$, and the set 
$\{d_i\}_{0\leq i\leq p^s-1}$ is a $k$-basis of $HH^1(kP)$. 
If $A=$ $\OP$, then by the previous example, $HH^1_{\OP}(kP)$ is a 
finite (possibly trivial) subgroup of $HH^1(kP)$.
By contrast, if $A=$ $\CO[x]/(x^{p^s})$, then $HH^1_A(kP)$ is the 
$(p^s-1)$-dimensional $k$-subspace spanned by the set 
$\{d_i\}_{1\leq i\leq p^s-1}$. Indeed, in that case, for
$1\leq i\leq p^s-1$ and $\lambda\in$ $\CO$ there is an
$\CO$-algebra automorphism $\alpha_{i,\lambda}$ in $\Aut_1(A)$ sending 
$x$ to $x+\pi \lambda x^i$, and this automorphism gives rise to the 
derivation $\bar\lambda d_i$ on $kP$, where $\bar\lambda$ is the image of 
$\lambda$ in $k$.
\end{Example}



\begin{thebibliography}{WWW}

\bibitem{AKO} M. Aschbacher, R. Kessar, and B. Oliver {\em Fusion 
Systems in Algebra and Topology}, London Math. Soc. Lecture Notes 
Series {\bf 391}, Cambridge University Press (2011).

\bibitem{BenII}
D. J. Benson, {\em Representations and Cohomology, Vol. II: Cohomology
of groups  and modules}, Cambridge studies in advanced mathematics
{\bf 31},  Cambridge University  Press (1991).

\bibitem{BroueEq} M. Brou\'e, {\em Equivalences of Blocks of Group
Algebras}, in: {Finite dimensional algebras and related
topics}, Kluwer (1994), 1--26.

\bibitem{BroueHigman}   M. Brou\'e, {\em Higman's criterion revisited},
Michigan Math. J. \textbf{58} (2009), 125--179.

\bibitem{BrPuloc} M. Brou\'e and L. Puig, {\em Characters and Local
Structure in $G$-Algebras}, J. Algebra {\bf 63} (1980), 306--317.

\bibitem{Brow} W. Browder, {\em Torsion in $H$-spaces}, Ann. Math
{\bf 74} (1961), 24--51.

\bibitem{CravenBook}
D.~A. Craven, \textit{The Theory of Fusion Systems}, Cambridge Studies 
in Advanced Mathematics, Vol. 131, Cambridge University Press, 
Cambridge, 2011

\bibitem{CR2} C. W. Curtis and I. Reiner, {\it Methods of
Representation theory} Vol. II, John Wiley and Sons,
New York, London, Sydney (1987).

\bibitem{FGM} 
 D. R. Farkas, C. Geiss, E. N. Marcos, {\em Smooth automorphism 
group schemes}. Representations of algebras (S\~ao Paulo, 1999), 
71--89,  Lecture Notes in Pure and Appl. Math., {\bf 224}, Dekker, 
New York (2002). 

\bibitem{Ger1} M. Gerstenhaber, {\em On the deformations of rings
and algebras}, Ann. Math. {\bf 79} (1964), 59--103.

%
%

\bibitem{KeLiNa} R. Kessar, M. Linckelmann, and G. Navarro, {\em A
characterisation of nilpotent blocks}, preprint (2014).

\bibitem{Listable} M. Linckelmann, {\em Stable equivalences of Morita
type for self-injective algebras and $p$-groups}, Math. Z. {\bf 223}
(1996) 87--100.

\bibitem{Lintransfer} M. Linckelmann,
{\em Transfer in Hochschild cohomology of blocks of finite groups},
Algebras Representation Theory \textbf{ 2} (1999), 107--135.

\bibitem{LinSol} M. Linckelmann, {\em Simple fusion systems and
the Solomon $2$-local groups}, J. Algebra \textbf{296}
(2006) 385--401.

\bibitem{Lifusintro} M. Linckelmann, {\em Introduction to Fusion 
Systems}, in: Group Representation Theory (edts. M. Geck, D. Testerman, 
J. Th\'evenaz), EPFL Press, Lausanne (2007), 79--113.

\bibitem{Mat} H. Matsumura, {\em Integrable derivations}, Nagoya Math.
J. {\bf 87} (1982), 227--245.

\bibitem{NaMa} L. Narv\'aez-Macarro, {\em On the modules of 
$m$-integrable derivations in non-zero characteristic.} Adv. Math. 
{\bf 229} (2012), 2712--2740.

\bibitem{Punil} L. Puig, {\em Nilpotent blocks and their source 
algebras}, Invent. Math. {\bf 93}  (1988), 77--116.

\bibitem{Pui00} L.~Puig \textit{The hyperfocal subalgebra of a block},
Invent. Math. 141 (2000), 365--397.

\bibitem{Rickcyclic} J. Rickard, {\em Derived categories and stable 
equivalence}, J. Pure Applied Algebra \textbf{61} (1989), 303--317.

\bibitem{RickDer} J. Rickard, {\em Derived equivalences as derived
functors}, J. London Math. Soc. \textbf{43} (1991), 37--48.

\bibitem{Rob08} G.~R.~Robinson, \textit{ On the focal defect group of a
block, characters of height zero, and lower defect group
multiplicities.} J. Algebra 320 (2008), no. 6, 2624–-2628.

\bibitem{RoSc} K. Roggenkamp and L. Scott, {\em Isomorphisms of
$p$-adic group rings}, Ann. Math. {\bf 126} (1987), 593--647.

\bibitem{Weibel} C. A. Weibel, {\em An introduction to homological
algebra}. Cambridge Studies Adv. Math. {\bf 38} (1994), Cambridge
University Press.

\end{thebibliography}
\end{document}